\documentclass[12pt,leqno,letterpaper]{article}

\usepackage{amsmath,amsthm,enumerate,amssymb}
\usepackage[utf8]{inputenc}

\usepackage[T1]{fontenc}
\usepackage[english]{babel}
\usepackage{graphicx}

\newtheorem{theorem}{Theorem}[section]

\newtheorem{definition}[theorem]{Definition}

\newtheorem{remark}[theorem]{Remark}
\newtheorem{example}[theorem]{Example}

\newcommand{\tr}{{\rm Tr\hskip -0.2em}~}

\begin{document}

\title{Non-commutative perspectives}
\author{Edward Effros and Frank Hansen}
\date{September 30, 2013}

\maketitle

\begin{abstract}

We prove that the non-commutative perspective of an operator convex function is the unique extension of the corresponding commutative perspective that preserves homogeneity and convexity. 

\end{abstract}

\section{Introduction  and preliminaries}

Let $ f $ be a function defined in the positive (open) half-line. The perspective function $ \mathcal P_f $ is the function of two variables given by
\[
P_f(t,s)=s f(t s^{-1})\qquad t,s>0.
\]
Depending on the application, we may also consider the function $ (t,s)\to\mathcal P_f(s,t) $ and denote this as the perspective of $ f. $ 

If $ A $ and $ B $ are commuting positive definite matrices then the matrix $ \mathcal P_f(A,B) $ is well-defined by the functional calculus.
Even if $ A $ and $ B $ do not commute we may by choosing an appropriate ordering define the perspective by setting
\[
\mathcal P_f(A,B)=B^{1/2}f(B^{-1/2}AB^{-1/2})B^{1/2}.
\]
This expression is well-defined and coincides with $ P_f(A,B), $ when $ A $ and $ B $ commute.

Effros  \cite{kn:effros:2009:1} only considered the case where each pair in the argument of the perspective function consists of commuting operators and proved in this way that the perspective of an operator convex function is operator convex as a functions of two variables.
Ebadian et. al. \cite{kn:ebadian:2011} noticed that virtually the same proof applies without any commutativity conditions. We include the proof for the convenience of the reader.

\begin{theorem}
Let $ f  $ be an operator convex function defined in the positive half-line.
The mapping
\[
(A,B)\to \mathcal P_f(A,B),
\]
defined in pairs of positive definite matrices, is convex.
\end{theorem}

\begin{proof} Consider positive definite matrices $ A_1,A_2 $ and $ B_1,B_2 $ and take a real number $ \lambda\in[0,1]. $ We set
\[
A=\lambda A_1+(1-\lambda) A_2\qquad\text{and}\qquad B=\lambda B_1+(1-\lambda) B_2.
\]
The matrices
\[
X=(\lambda B_1)^{1/2}B^{-1/2}\qquad\text{and}\qquad Y=((1-\lambda)B_2)^{1/2}B^{-1/2}
\]
satisfy
\[
X^*X+Y^*Y=B^{-1/2}\lambda B_1B^{-1/2}+B^{-1/2}(1-\lambda)B_2B^{-1/2}=1
\]
and
\[
\begin{array}{l}
X^* B_1^{-1/2}A_1B_1^{-1/2} X+Y^* B_2^{-1/2}A_2 B_2^{-1/2}Y\\[1ex]
=B^{-1/2} \lambda A_1B^{-1/2} +B^{-1/2} (1-\lambda)A_2B^{-1/2}=
B^{-1/2}  A B^{-1/2}.
\end{array} 
\]
Hence
\[
\begin{array}{l}
\mathcal P_f(\lambda A_1+(1-\lambda)A_2, \lambda B_1+(1-\lambda) B_2)=\mathcal P_f(A,B)\\[1.5ex]
=B^{1/2}f(B^{-1/2}AB^{-1/2})B^{1/2}\\[1.5ex]
=B^{1/2} f\bigl(X^* B_1^{-1/2}A_1B_1^{-1/2} X+Y^* B_2^{-1/2}A_2 B_2^{-1/2}Y\bigr) B^{1/2}\\[1.5ex]
\le B^{1/2}\big( X^* f(B_1^{-1/2}A_1B_1^{-1/2})X+Y^*f(B_2^{-1/2}A_2B_2^{-1/2})Y\bigr) B^{1/2}\\[1.5ex]
=(\lambda B_1)^{1/2} f(B_1^{-1/2}A_1B_1^{-1/2})(\lambda B_1)^{1/2}\\[1ex]
\hskip 8em +\,((1-\lambda)B_2)^{1/2}f(B_2^{-1/2}A_2B_2^{-1/2})((1-\lambda)B_2)^{1/2}\\[1ex]
=\lambda\mathcal P_ f(A_1,B_1) + (1-\lambda) \mathcal P_f(A_2,B_2),
\end{array}
\]
where we used Jensen's operator inequality \cite{kn:hansen:1982}.
\end{proof}

There are obvious similarities between the notion of a perspective function and the operator means studied by Kubo and Ando \cite{kn:kubo:1980}. The crucial difference is that the representing function of an operator mean is operator monotone (and hence operator convex). We are considering operator convex functions, and they are in general not monotone.

\begin{example}
Consider the operator convex function $ f(t)=-\log t $ defined for $ t>0. $ The (classical) perspective function is given by
\[
P_f(t,s)=-t\log(t^{-1}s)=t\log t-t\log s,
\]
and the non-commutative perspective is then given by
\[
\mathcal P_f(A,B)=-A^{1/2}\log(A^{-1/2}BA^{-1/2})A^{1/2}. 
\]
The relative entropy $ S(A,B) $ is defined by setting
\[
S(A,B)=\tr A\log A-\tr A\log B
\]
and is known to be a convex function of two variables.  For commuting matrices we have
\[
S(A,B)=\tr\mathcal P_f(A,B),
\]
although the two quantities in general are different.  
\end{example}

\subsection{Spectral functions}

Let $ B(\mathcal H) $ denote the set of bounded linear operators on a Hilbert space $ \mathcal H. $ A function $ F\colon\mathcal D\to B(\mathcal H) $ defined in a convex domain $ \mathcal D $ of normal operators in $ B(\mathcal H) $ is called a spectral function, if  it can be written on the form $ F(x)=f(x) $ for some real or complex function $ f $ defined in a real interval $ I, $ where $ f(x) $ is obtained by applying the functional calculus for normal operators.

Although this definition appears quite intuitive it contains some hidden assumptions. Firstly, the domain $ \mathcal D $ should be invariant under unitary transformations and
\begin{equation}\label{unitary invariance}
F(u^*xu)=u^*F(x)u\qquad x\in\mathcal D
\end{equation}
for every unitary transformation $ u $ on $ \mathcal H. $ Secondly, for orthogonal projections $ p $ and $ q $ on $ \mathcal H, $ the element $ pxp+qxq\in\mathcal D $ 
and
\begin{equation}\label{rule for block matrices}
F(pxp+qxq)=pF(pxp)p+qF(qxq)q
\end{equation}
for arbitrary $ x\in B(\mathcal H) $ such that $ pxp $ and $ qxq $ are contained in $ \mathcal D. $ An operator function $ x\to F(x) $ is a spectral function if and only if (\ref{unitary invariance}) and (\ref{rule for block matrices}) are satisfied, cf. \cite{kn:davis:1957, kn:hansen:2003:2}.

\section{The main result}

The notion of spectral function is not immediately extendable to functions of two variables. However, we may consider the two properties of spectral functions noticed by C. Davis as a kind of regularity conditions, and they are readily extendable to functions of more than one variable. 
\begin{definition}
Let $ F\colon\mathcal D\to B(\mathcal H) $ be a function of two variables defined in a convex domain $ \mathcal D\subseteq B(\mathcal H)\times B(\mathcal H). $
We say that $ F $ is regular if

\begin{enumerate}[(i)]

\item The domain $ \mathcal D $ is invariant under unitary transformations of $ \mathcal H $ and
\[
F(u^*xu, u^*yu)=u^* F(x,y) u\qquad (x,y)\in\mathcal D
\]
for every unitary $ u $ on $ \mathcal H. $

\item Let $ p $ and $ q $ be orthogonal projections on $\mathcal H. $ Then the pair of diagonal block matrices $ (pxp+qxq, pyp+qyq)\in\mathcal D $ and 
\[
F(pxp+qxq, pyp+qyq)=pF(pxp,pyp)p+qF(qxq,qyq)q
\]
for arbitrary $ x,y\in B(\mathcal H) $ such that $ (pxp,pyp) $ and $ (qxq,qyq) $ are in $ \mathcal D. $
\end{enumerate}                
\end{definition}
The following theorem is related to \cite[Theorem 2.2]{kn:hansen:1983}.

\begin{theorem}
Let $ (A,B)\to F(A,B) $ be a regular map from pairs of bounded positive semi-definite operators on an infinite dimensional Hilbert space $ \mathcal H $ into $ B(\mathcal H) $ satisfying the conditions:

\begin{enumerate}[(i)]

\item $ F(tA,tB)=t F(A,B) \qquad t>0 $

\item $\displaystyle F\left(\frac{A_1+A_2}{2},\frac{B_1+B_2}{2}\right)\le\frac{F(A_1,B_1)+F(A_2,B_2)}{2} $

\item $ F(0,0)=0, $ and $ B\to F(1,B) $ is continuous on bounded subsets in the strong operator topology, where $ 1 $ denotes the unit operator on $ \mathcal H. $

\end{enumerate}
Then there exists an operator convex function  $ f\colon\mathbf R_+\to\mathbf R $ such that
\[
F(1,t\cdot 1)=f(t) 1\qquad t>0.
\]
Furthermore,
\[
F(A,B)=A^{1/2} f(A^{-1/2}BA^{-1/2}) A^{1/2} =\mathcal P_f(A,B)
\]
for positive definite invertible operators $ A $ and $ B. $
\end{theorem}

\begin{proof}
The regularity of $ F $ entails that 
\[
u^*F(1,t\cdot 1)u=F(1,t\cdot 1)\qquad t>0
\]
for every unitary $ u $ in $ B(\mathcal H). $  Thus $ F(1,t\cdot 1) $ commutes with every unitary in $ B(\mathcal H) $ and is therefore of the form
\[
 F(1,t\cdot 1)=f(t)\cdot 1\qquad t>0
\]
for some function $ f\colon\mathbf R_+\to\mathbf R. $ If $ A=\sum_{i=1}^n \lambda_i P_i $ is the spectral decomposition of a finite rank positive definite operator $ A $ on $ \mathcal H $ then
\begin{equation}\label{formula for finite rank operators}
\begin{array}{rl}
F(1,A)&=\displaystyle\sum_{i=1}^n P_i F(P_i,\lambda_i P_i)P_i\\[3ex]
&=\displaystyle\sum_{i=1}^n P_i F(1,\lambda_i\cdot 1)P_i\\[3ex]
&=\displaystyle\sum_{i=1}^n f(\lambda_i)P_i=f(A)
\end{array}
\end{equation}
by the regularity of $ F. $ Since $ F $ is mid-point convex it follows that $ f $ is mid-point operator convex and thus operator convex\footnote{It is a curiosity that continuity is not required to prove that mid-point operator convexity implies operator convexity.}.

Let now $ C $ be a contraction and consider the unitary block matrices
\[
U=\displaystyle\begin{pmatrix}
      C & (1-CC^*)^{1/2}\\
      (1-C^*C)^{1/2} & -C^*
      \end{pmatrix}
=\begin{pmatrix}
  C & D\\
  E & -C^*
  \end{pmatrix}
\]
and
\[
  V=\displaystyle\begin{pmatrix}
  C & -D\\
  E & C^*
  \end{pmatrix}.
\]      
It is plain to calculate that
\[
\frac{1}{2}U^*\begin{pmatrix}
                         A & 0\\
                         0  & 0
                         \end{pmatrix}U
+\frac{1}{2}V^*\begin{pmatrix}
                         A & 0\\
                         0 & 0
                         \end{pmatrix}V
 =\begin{pmatrix}
                C^*AC & 0\\
                0           & DAD
                \end{pmatrix} .
 \]     
 We then obtain
 \[
 \begin{array}{l}
 \displaystyle\begin{pmatrix}
 C^* F(A,B) C & 0\\
 0                      & D F(A,B) D
 \end{pmatrix}\\[3ex]
 =\displaystyle\frac{1}{2} U^*\begin{pmatrix}
                                                    F(A,B) & 0\\
                                                    0           & 0
                                                    \end{pmatrix}U
 +\frac{1}{2} V^*\begin{pmatrix}
                                                    F(A,B) & 0\\
                                                    0           & 0
                                                    \end{pmatrix}V\\[3ex]
=\displaystyle\frac{1}{2} U^* F\left(\begin{pmatrix}
                                                    A & 0\\
                                                    0 & 0
                                                    \end{pmatrix}, \begin{pmatrix}
                                                    B & 0\\
                                                    0 & 0
                                                    \end{pmatrix}\right)U
 +\frac{1}{2}VF\left(\begin{pmatrix}
                                                    A & 0\\
                                                    0 & 0
                                                    \end{pmatrix}, \begin{pmatrix}
                                                    B & 0\\
                                                    0 & 0
                                                    \end{pmatrix}\right)V^*\\[3ex]
=\displaystyle\frac{1}{2} F\left(U^*\begin{pmatrix}
                                                    A & 0\\
                                                    0 & 0
                                                    \end{pmatrix}U, U^*\begin{pmatrix}
                                                    B & 0\\
                                                    0 & 0
                                                    \end{pmatrix}U\right)
 +\frac{1}{2} F\left(V^*\begin{pmatrix}
                                                    A & 0\\
                                                    0 & 0
                                                    \end{pmatrix}V, V^*\begin{pmatrix}
                                                    B & 0\\
                                                    0 & 0
                                                    \end{pmatrix}V\right)\\[3ex]
\ge\displaystyle F\left(\frac{1}{2} U^*\begin{pmatrix}
                                                    A & 0\\
                                                    0 & 0
                                                    \end{pmatrix}U+ \frac{1}{2} V^*\begin{pmatrix}
                                                    A & 0\\
                                                    0 & 0
                                                    \end{pmatrix}V, \,
\frac{1}{2} U^*\begin{pmatrix}
                                                    B & 0\\
                                                    0 & 0
                                                    \end{pmatrix}U+ \frac{1}{2} V^*\begin{pmatrix}
                                                    B & 0\\
                                                    0 & 0
                                                    \end{pmatrix}V\right)\\[3ex]
=F\left(\begin{pmatrix}
                          C^* AC & 0\\
                          0            & DAD
                          \end{pmatrix},\,
\begin{pmatrix}
                          C^* BC & 0\\
                          0            & DBD
                          \end{pmatrix}\right)\\[3ex]
=\displaystyle\begin{pmatrix}
                          F(C^*AC,C^*BC) & 0\\
                          0                              & F(DAD,DBD)   
                          \end{pmatrix},
\end{array}
 \]
 where, in the second equality, we used $ F(0,0)=0 $ from condition $ (iii). $ In particular, we have proved that
 \begin{equation}\label{transformer inequality}
 C^* F(A,B) C\ge F(C^*AC, C^* BC)
 \end{equation}
 for contractions $ C. $ However, the homogeneity of $ F $ then implies (\ref{transformer inequality}) for any operator $ C. $ In particular, if $ C $ is invertible we obtain
 \[
 F(A,B)\ge (C^*)^{-1} F(C^*AC, C^* BC) C^{-1}\ge F(A,B),
 \]
 hence there is equality and thus
 \[
 C^* F(A,B) C= F(C^*AC, C^* BC).
 \]
 For invertible $ A $ we therefore obtain
 \[
 A^{-1/2} F(A,B) A^{-1/2} = F(1,A^{-1/2} B A^{-1/2}).
 \] 
 If $ B $ is positive definite and of finite rank, then so is $ A^{-1/2} B A^{-1/2} $ and thus
 \[
 F(1,A^{-1/2} B A^{-1/2}) =f(A^{-1/2} B A^{-1/2})
 \] 
 by equation (\ref{formula for finite rank operators}). 
 
Let $ g $ be a continuous function defined in an open interval $ I. $ The functional calculus $ X\to g(X) $ is strongly continuous on bounded subsets of self-adjoint operators $ X $ with spectra in $ I, $ cf. the proof of 
 \cite[Lemma 2.2]{kn:bendat:1955}. Indeed, if $ (X_i) $ is a bounded net of operators converging strongly to $ X, $ then the inequality
 \[
 \begin{array}{rl}
 \|X^k\xi-X_i^k \xi\| &\le \|X^k\xi-X_i^{k-1}X\xi\|+\|X_i^{k-1}X\xi - X_i^k\xi \|\\[1.5ex]
 &\le \|X^k\xi-X_i^{k-1}X\xi\|+\|X_i^{k-1}\|\cdot \|X\xi - X_i\xi\|,
 \end{array}
 \]
together with an induction argument, shows that $ (X_i^k) $ converges strongly to $ X^k $ for any natural number $ k. $ The assertion then follows by approximating $ g $ uniformly by polynomials in a compact subset of $ I $ containing the spectrum of $ X. $ The continuity condition in $ (iii) $ therefore implies
 \[
 F(1,A^{-1/2} B A^{-1/2}) =f(A^{-1/2} B A^{-1/2})
 \] 
 and thus
 \[
 F(A,B)= A^{1/2} f(A^{-1/2} B A^{-1/2}) A^{1/2}
 \]
 for arbitrary positive definite operators $ A $ and $ B $ defined on $ \mathcal H. $
 \end{proof}

 \begin{remark}
It is crucial in the above proof that the regular map $ F(A,B) $ is defined for positive semi-definite operators. We are therefore excluding the limiting case, 
 \[
 F(A,B)=A^{1/2} f(A^{-1/2} B A^{-1/2}) A^{1/2}=AB^{-1}A,
 \]
that appears by setting $ f(t)=t^{-1} $ for $ t>0. $  
 \end{remark}
 
 Notice that the above theorem has an obvious counterpart if convexity is replaced by concavity.
 The theorem states that a non-commutative perspective function, that allows an extension to positive semi-definite operators, is the unique extension of a commutative perspective function to a homogeneous, convex and regular operator mapping. In particular, the geometric operator mean 
\[
A\#B=A^{1/2}\bigl(A^{-1/2}BA^{-1/2}\bigr)^{1/2}A^{1/2}
\] 
is the only sensible extension of the geometric mean $ (t,s)\to\sqrt{ts} $ of positive numbers to a homogeneous and concave operator mapping.

{\small
 

\vfill

\noindent Edward Effros: Mathematics Department, UCLA, Los Angeles, CA 90015.\\
Email: ege@math.ucla.edu\\[1ex]
\noindent Frank Hansen: Institute for International Education, Tohoku University, Japan.\\
Email: frank.hansen@m.tohoku.ac.jp.
      }

\end{document}